\documentclass[12pt,oneside]{amsart}
     \usepackage{amssymb}
    \usepackage{amsmath}
    \usepackage[english]{babel}
    \usepackage[autostyle]{csquotes}
    \usepackage{hyperref}

      \theoremstyle{plain}
      \newtheorem{theorem}{Theorem}[section]
      \newtheorem{lemma}[theorem]{Lemma}
      \newtheorem{corollary}[theorem]{Corollary}

      \theoremstyle{definition}
      \newtheorem{definition}[theorem]{Definition}
      
      \newtheorem{remark}{Remark}
     \newtheorem*{notation}{Notation}

      \newcommand{\nil}{\emptyset}
      \newcommand{\g}{\mathfrak g}
      \newcommand{\h}{\mathfrak{h}}

       \newcommand{\E}{\mathcal{E}}
       \newcommand{\F}{\mathcal{F}}
       \renewcommand{\S}{\mathcal{S}}
       \newcommand{\N}{\mathbb{N}}
        \newcommand{\R}{\mathcal{R}}

       \newcommand{\A}{\mathcal{A}}
       \newcommand{\B}{\mathcal{B}}

     \newcommand{\Con}{\operatorname{Con}}
     \newcommand{\atom}[1]{\textbf{atom}(\mathcal{#1})}
     \newcommand{\n}[1]{\textbf{n}(#1)}
      \newcommand{\con}[2]{\Con(\mathfrak{#1},\mathcal{#2})}
       
      \newcommand{\tcon}[2]{\Con_{\text{t}}(\mathfrak{#1},\mathcal{#2})}

      \newcommand{\Cl}[1]{\text{Cl}(#1)}
      \newcommand{\Ncp}[2]{(\hat{\mathfrak{#1}},\hat{\mathcal{#2}})}
      \newcommand{\word}[3]{W(#1,#2;\mathfrak{#3})}

      \makeatletter
      \def\@setcopyright{}
      \def\serieslogo@{}
      \makeatother
   
\begin{document}

   \author{Ali Rejali}
   \address{University of Isfahan}  
   \curraddr{Department of Mathematics, Faculty of  Sciences , University of Isfahan, Isfahan, 81746-73441, Iran}
   \email{rejali@sci.ui.ac.ir}


   \author{Meisam Soleimani Malekan}
   \address{Department of Mathematics, Ph. D student, University of Isfahan, Isfahan, 81746-73441, Iran}
   \email{m.soleimani@sci.ui.ac.ir}

   \title{Configuration Equivalence is not Equivalent to Isomorphism}

     \begin{abstract}
Giving a condition for the the amenability of groups, Rosenblatt and Willis, first introduced the concept of configuration. From the beginning of the theory, the question whether the concept of configuration equivalence coincides with the concept of group isomorphism was posed. We negatively answer this question by introducing two non-isomorphic, solvable and hence amenable groups which are configuration equivalent. Also, we will prove this conjecture, due to Rosenblatt and Willis, that configuration equivalent groups, both include the free non-Abelian group of same rank or not. We show that two-sided equivalent groups have same class numbers.
   \end{abstract}



   \thanks{The first authors would like to express  their gratitude toward Banach Algebra Center of Excellence for Mathematics, University of Isfahan. We also thank Yves de Cornulier for the ideas in the proof of Theorem \ref{non-isomorphic generatin-bijective groups}. }
   \thanks{}

   \dedicatory{}

   \date{\today}


   \maketitle



   \section{Introduction and Definitions}
   In this paper, all groups were assumed to be finitely generated and discrete. Let $G$ be a group, we denote the identity of the group $G$ by $e_G$. Given any finite (ordered) subset $\g =(g_1,\dotsc, g_n)$ of $G$, there is a Cayley graph denoted by $\Gamma:=\Gamma(G,\mathfrak{g})$, with vertices being the elements of $G$ and the directed edges being from $g$ to $g_ig$ for $g\in G$ and $i\in\{1,\dotsc,n\}$. 
\par  let $G$ be a group, $\mathfrak{g} =(g_1,\dotsc, g_n)$ be an ordered generating set  and $\Gamma=\Gamma(G,\mathfrak{g})$ be its corresponded Cayley graph. Assume that $\E =\{E_1,\dotsc,E_m\}$ is a finite partition of $G$, which can be considered as a coloring of $\Gamma$ by m colors. 
  \par A \textit{configuration} $C$ is an $(n+1)$-tuple of colors $(c_0,\dotsc, c_n)$ with each $c_k$ being one of the $m$ colors, and there are $x_0, x_1,\dotsc, x_n\in G$ such that $x_i$ is the color $c_i$, $i=0,1,\dotsc, n$, and for each $i$, $x_i=g_ix_0$. In this case, we may say that $(x_0,x_1,\dotsc,x_n)$ has the configuration $C$. Simultaneously considering the right multiplication, the concept of two-sided configuration is reached; A \textit{two-sided configuration} is a $(2n + 1)$-tuple $C= (c_0,c_1,\dotsc,c_{2n})$ satisfying $c_i\in\{1,\dotsc,m\}$, $i=0,1,\dotsc,2n$, and there exists $x\in E_{c_0}$ such that $g_ix\in E_{c_i}$ and $xg_i\in E_{c_{i+n}}$ for each $i\in\{1,\dotsc,n\}$. \par
First, the concept of configuration was introduced by Rosenblatt and Willis to give a characterization for the amenability of groups and then, to characterize normal sets, the concept of two-sided configuration was suggested.
\par  For $\g$ and $\E$ as above, we called $(\g,\E)$ a \textit{configuration pair}. The set of configurations (two-sided configurations, resp.) corresponding to the configuration pair $(\g,\E)$ will be denoted by $\con gE$ ($\tcon gE$, resp.). The set of all configuration and two-sided configuration sets of $G$, denoted by $\Con(G)$ and $\Con_\text t(G)$, respectively. 
\par The origin of the theory of configuration, is the conjecture raised in \cite{rw}, that the combinatorial properties of configurations can be used to characterize various kinds of behavior of groups (like the group being Abelian or the group containing a non-Abelian free subgroup). This conjecture leads to the notion of (two-sided) configuration equivalence;  A group $G$ is \textit{configuration contained} in a group $H$, written $G\precsim H$, if $\Con(G)\subseteq\Con(H)$, and two groups $G$ and $H$ are \textit{configuration equivalent}, written $G\approx H$, if $\Con(G) =\Con(H)$. The concepts of being \textit{two-sided configuration contained}, and \textit{two-sided configuration equivalent} are similarly defined, the notations used denoting these concepts are  \enquote{$\precsim_\text t$} and \enquote{$\approx_\text t$}. 

 In \cite{arw}, the first steps of the theory were taken. It was shown there that finiteness and periodicity are the properties which can be characterized by configuration. In that paper, the authors proved that for two configuration equivalent groups, the isomorphism classes of their finite quotients were the same. The finite index property can be extended to Abelian quotient property (see \cite{ary}). Also, it was shown in \cite{arw} that two configuration equivalent groups, should satisfy in the same semi-group laws, and in \cite{rs}, this result was generalized by proving that same group laws should be established in configuration equivalent groups. Hence, in particular, being Abelian and the group property of being nilpotent of class $c$ are, in particular, another properties which can be characterized by configuration (see \cite{arw} and \cite{ary}). In \cite{ary}, it was shown that if $G\approx H$, and $G$ is a torsion free nilpotent group of Hirsch length $h$, then so is $H$. It was interesting to know the answer to the question whether being FC-group is conserved by equivalence of configuration. In \cite{ary}, this question was answered under the assumption of being-nilpotent. In \cite{rs}, this question was affirmatively answered without any extra hypothesis. In addition, it was shown in that paper that the solubility of a group $G$ can be recovered from $\Con G$. In our recent paper, \cite{rs1}, we showed that the notion of normality can be obtained from two-sided configuration equivalence. Also, in the presence of normality, we showed that if $G$ and $H$ are two-sided equivalent, and if $G$ has a normal subgroup $N$, which the quotient $G/N$ is finitely presented and $G$ has a \enquote{recognizable} configuration pair w.r.t. $N$, then $H$ contains a normal subgroup $\mathfrak{N}$, such that $G/N\cong H/\mathfrak{N}$. It was shown, that the class of \enquote{polynomial type} groups -- involving finite, Abelian, free and polycyclic groups -- satisfied the \enquote{recognizability} condition.  
 
\par We also interested in investigating the question: For which subclasses of the class \textbf{G} of all groups, does configuration equivalence coincide with isomorphism? In \cite{arw}, this question was answered positively for the class of finite, free and Abelian groups. In \cite{ary}, it was shown that those groups with the form of $\mathbb{Z}^n\times F$, where $\mathbb Z$ is the group of integers, $n\in\mathbb N$ and $F$ is an arbitrary finite group, are determined up to isomorphism by their configurations. In \cite{ary}, it was proved that if $G\approx D_\infty$, where $D_\infty$ is the infinite dihedral group, then $G\cong D_\infty$.
 \par In \cite{rs}, we pointed out that it was the existence of \enquote{golden} configuration pairs which implied isomorphism. Indeed, we showed that, in the class of finitely presented Hopfian groups with golden configuration pair, configuration equivalence coincided with isomorphism.
 \par In the light of two-sided configuration, it was proved in \cite{rs1} that, for \enquote{polynomial type} groups, and for groups with finite commutator subgroup, the \enquote{$\approx_\text t$} and \enquote{$\cong $} are equivalent. Specially, for polycyclic or FC--groups, these two notation are equivalent. 
 \par We will prove that if $G$ and $H$ are two configuration equivalent groups, and if $G$ contains $\mathbb F_n$, the non-Abelian free group of rank $n$, then $H$, also, contains $\mathbb F_n$.
 \par In the present paper, we define configuration sets for a finite sigma algebras of a group, and with the help of them, we will show that two-sided equivalence groups, have the same number of normal subgroups of finite index $n$, $n\in\mathbb{N}$. Also, if $G\approx_\text tH$, then the cardinality of normal subgroups which their quotient are polycyclic is the same in both $G$ and $H$. The class number- the number of different conjugacy classes- will be shown to be equal for two-sided equivalent groups. In the class of two-sided equivalent groups which have a finite class number, we will study some type of subgroups, and will show that in this class the set of finite, and polycyclic subgroups should be the same.
  \par The question whether (two-sided) configuration equivalence implies isomorphism has been seen to be open since the beginning of the theory of configuration. We negatively answer this open question by presenting two non-isomorphic solvable and hence amenable groups, both having the same (two-sided) configuration sets. 
   Like our two recent papers (\cite{rs} and \cite{rs1}), we will use the following notation:
  \begin{notation}
  \label{notation1}
Let $G$ be a group with $\mathfrak{g}=(g_1,\dotsc,g_n)$ as its ordered subset. Let $p\in\mathbb{N}$, $J\in\{1,2,\dotsc,n\}^p$ and $\rho\in\{\pm1\}^p$. We denote the product 
$\prod_{i=1}^pg_{J(i)}^{\rho(i)}$ by $W(J,\rho;\mathfrak{g})$. We call the pair 
$(J,\rho)$ a \textit{representative pair} in $\g$ and $W(J,\rho;\mathfrak{g})$ a \textit{word} corresponding to $(J,\rho)$ in $\g$. 
\end{notation}
When we speak of a representative pair, $(J,\rho)$, we assume the same number of components for $J$ and $\rho$.  The number of components of $J$, denotes by $\n J$.

 \section{Configuration and Finite Sigma Algebras}
 What is really important in configuration is the image of subsets of a group $G$, under left translations by finite subsets of  $G$. So, it seems that for sets $\con gE$, or $\tcon gE$ of a group $G$, we can replace $E$ by a sigma algebra $\A$, and indeed we can. we involve sigma algebras in the theory of configuration as follows:\par 
 Let $G$ be a group. There is a correspondence between the finite sigma algebras of $G$, and finite partitions of $G$. Indeed, for a finite sigma algebra $\A$, the set of atoms of $\A$, $\E:=\{E_1,\dotsc,E_m\}$ is a partition of $G$. We denote the atomic sets of a sigma algebra $\A$ by $\atom A$. Also, if $\mathcal{C}$ is a finite collection of subsets of $G$, we use $\sigma(\mathcal C)$ to denote the sigma algebra generated by $\mathcal C$. In the following, we always consider sigma algebras to be finite. \par
Now, we try to rewrite our symbols in configuration for sigma algebras; For a sigma algebra $\A$, we define $\con gA$ to be $\Con(\g,\atom A)$. Similarly, we can define $\tcon gA$. Remember following efficient symbol from \cite{rs1}:
\begin{notation}
\label{notation2}
Let $G$ and $H$ be two groups with $\g$ and $\h$ as their generating set, respectively. Coloring Cayley graph $\Gamma_1:=\Gamma(G,\mathfrak{g})$ and $\Gamma_2:=\Gamma(H,\mathfrak{h})$ with same colors, we get partitions $\E$ and $\F$ of $G$ and $H$, respectively. 
For two sets $E\in E$ and $F\in\F$, we write $E\circledS F$ to show that we have two sets of the same color. In particular, If we have {\small $\con gE=\con hF$} or {\small $\tcon gE=\tcon hF$}, for configuration pairs $(\g,\E)$ and $(\h,\F)$ for groups $G$ and $H$, respectively; This implies that, their corresponded Cayley graphs $\Gamma_1:=\Gamma(G,\mathfrak{g})$ and $\Gamma_2:=\Gamma(H,\mathfrak{h})$, colored with same colors.
\end{notation}
We can also use $\circledS$ for sigma algebras; Let $\E:=\{E_1,\dotsc,E_m\}$ and $\F:=\{F_1,\dotsc,F_m\}$ be partitions of $G$ and $H$ resp. such that $E_i\circledS F_i$, $i=1,\dotsc,m$. For $A\in\sigma(\E)$ and $B\in\sigma(\F)$, say $A\circledS B$, when 
$$\{k:\,E_k\cap A\neq\nil\}=\{k:\, F_k\cap B\neq\nil\}$$
In other words, if $A\circledS B$, and $A=E_{i_1}\cup\dotsb\cup E_{i_j}$, then $B=F_{i_1}\cup\dotsb\cup F_{i_j}$. 
The following technical lemma will be used in the following:
\begin{lemma}
\label{important}
Let $G$ and $H$ be two groups with sigma algebras $\A$ and $\B$, resp. and $\con gA=\con hB$. Suppose $A_1,A_2\in\A$ and $B_1, B_2\in\B$, are such that $A_i\circledS B_i$, $i=1,2$, and $g\in\g$ corresponded to $h\in\h$. we have 
\begin{itemize}
\item[(1)] If $gA_1\subseteq A_2$, then $hB_1\subseteq B_2$,
\item[(2)] If $gA_1=A_2$, then $hB_1=B_2$
\end{itemize}
\end{lemma}
\begin{proof}
Set 
\begin{align*}
\g&=(g_1,\dotsc,g_n),\quad \atom{A}=\{E_1,\dotsc,E_m\}\\
\h&=(h_1,\dotsc,h_n),\quad \atom{B}=\{F_1,\dotsc,F_m\}
\end{align*}
Without loss of generality, let $g$ and $h$ match $g_1$ and $h_1$. Also, set 
$$I_k:=\{i: E_i\cap A_k\neq\nil\},\quad k=1,2$$
So, by assumptions, 
$$A_k=\bigcup_{i\in I_k}E_i,\quad B_k=\bigcup_{i\in I_k}F_i\quad(k=1,2)$$
Now, for $C=(c_0,c_1,\dotsc,c_n)$ in $\con gA$, $c_1\in I_2$ if $c_0\in I_1$, this proves (1).\par 
For proving (2), note that if $C=(c_0,c_1,\dotsc,c_n)$ in $\con gA$, then $c_0\in I_1$, if and only if $c_1\in I_2$.
\end{proof}
If $\tcon gE=\tcon hF$, one can easily get an analog of the above lemma for left multiplication replaced with right multiplication.
\par Let $G$ and $H$ be two groups. Consider partitions $\R$ and $\S$ of $G$ and $H$ respectively, and their refinements, $\R'$ and $\S'$. Assume that 
\begin{align*}
\R&=\{R_1,\dotsc,R_l\}, \quad \R'=\{R'_1,\dotsc,R'_k\}\\
\S&=\{S_1,\dotsc,S_l\},\quad \S'=\{S'_1,\dotsc,S'_k\}\\
\end{align*}
 We may say that these two pairs $(\R',\R)$ and $(\S',\S)$ are \textbf{similar} and may write 
$(\R',\R)\sim(\S',\S)$, if 
$$\{i:R'_j\cap R_i\neq\nil\}=\{i:S'_j\cap S_i\neq\nil\}$$
for each $j=1,\dotsc,k$. Also, for sigma algebras $\A$ and $\B$, resp., and their sigma subalgebras, $\A'$ and $\B'$, resp. we say $(\A,\A')$ and $(\B,\B')$ are \textbf{similar}, written $(\A,\A')\sim(\B,\B')$, if 
$$(\atom A,\atom{\A'})\sim(\atom{B},\atom\B')$$
Now, we can rewrite \cite[Lemma 3.2]{rs} in sigma algebras:
\begin{lemma}\label{partitions in sigma algebras}
Let $\A$ and $\B$ be sigma algebras of groups $G$ and $H$, resp. Suppose $\g$ and $\h$ are generating sets of $G$ and $H$ such that $\con gA=\con hB$. If $\A'$ and $\B'$ are sigma sub-algebras of $\A$ and $\B$, such that $(\A,\A')\sim(\B,\B')$, then $\con g{A'}=\con h{\B'}$. 
\end{lemma}
The sigma algebraic version of \cite[Lemma 3.2]{rs1} is similarly obtained. 
\par By use of sigma algebras, equality of Tarski numbers of configuration equivalent groups, is easily obtained. Recall the following definition of paradoxical decomposition and Tarski number:
 \begin{definition}\label{Tarski numbers}
 A group $G$ admits a \textit{paradoxical decomposition} if there exist positive integers $m$ and $n$, disjoint subsets $P_1$, . . . , $P_m$, $Q_1$, . . . , $Q_n$ of $G$ and elements $x_1$, . . . , $x_m$, $y_1$, . . . , $y_n$ of $G$ such that
 $$ G=\bigcup_{i=1}^m x_iP_i=\bigcup_{j=1}^ny_jQ_i$$
 The minimal possible value of $m + n$ in a paradoxical decomposition of $G$ is the \textit{Tarski number} of $G$ and denoted by $\tau(G)$. If a group $G$ doesn't have a paradoxical decomposition, it means that $G$ is amenable; In this case we will define $\tau(G)$ to be $\infty$.\par
 \end{definition}
 \begin{theorem}
\label{Equality of tarski numbers}
Let groups $G$ and $H$ be configuration equivalent. Then $\tau(G)=\tau(H)$.
\end{theorem}
\begin{proof}
If $G$ is amenable, so is $H$, and we have done. Hence without loss of generality, suppose that $\tau(G)$ is finite. Fix a generating set of $G$, say $\g_0$. Let $m$ and $n$, $(P_1,\dotsc,P_m)$, $(Q_1,\dotsc,Q_n)$ and elements $\mathfrak x:=(x_1,\dotsc,x_m)$, $\mathfrak y:=(y_1,\dotsc,y_n)$ be as in Definition \ref{Tarski numbers}. Consider the sigma algebra $\A$, generated by sets
$$P_i,\, x_rP_i,\, Q_j,\, y_sQ_j,\quad i,r=1,\dotsc,m,\, j,s=1,\dotsc,n$$
Set $\g:=\g_0\oplus\mathfrak{x}\oplus\mathfrak{y}$. Since $G\approx H$, there exists a generating set $\h=\h_0\oplus\mathfrak{u}\oplus\mathfrak{v}$ and sigma algebra $\B$ of $G$, such that 
$$\con gA=\con hB$$
where $\h_0$, $\mathfrak{u}:=(u_1,\dotsc,u_m)$, $\mathfrak{v}:=(v_1,\dotsc,v_n)$ are corresponded to $\g_0$, $\mathfrak{x}$ and $\mathfrak{y}$, resp. Suppose that $C_i$, $D_j$ in $\B$ are such that $C_i\circledS P_i$ and $D_j\circledS Q_j$, $i=1,\dotsc,m$, $j=1,\dotsc,m$. Therefore, Lemma \ref{important} leads to $u_rC_i\circledS x_rP_i$ and $v_sD_j\circledS y_sQ_j$. So, $\mathfrak{u}$, $\mathfrak{v}$, $C_i$'s and $D_j$'s satisfy the conditions of Definition \ref{Tarski numbers}, and hence, $\tau(H)\leq\tau(G)$. By the symetry, we have also $\tau(G)\leq\tau(H)$, which completes the proof.
\end{proof}
We say that a group $G$ admits a $G=nG$ decomposition, $n\in\N$, if there exists disjoint sets 
\[ \{P_{i,j}: i=1,\dotsc,n,\,j=1,\dotsc,k(i)\}\]
along with elements $\mathfrak{x}^{(i)}=(x_{i,1},\dotsc,x_{i,k(i)})$, $i=1,\dotsc,n$, of $G$, such that 
$$ G=\bigcup_{j=1}^{k(i)}x_{i,j}P_{i,j}\quad(i=1,\dotsc,n)$$
So, paradoxical decomposition of a group $G$, is a $G=2G$ decomposition. If a group $G$, admits a $G=nG$ decomposition, we denote the minimal amount of $\sum_{i=1}^n k(i)$ by $\tau_n(G)$. Using the method of proof of the above theorem, we can easily show that
\begin{theorem}
Let $G$ and $H$ be two configuration equivalent group. If $G$ admits a $G=nG$ decomposition, $n\in\N$, then $H$ also admits a $H=nH$ decomposition, and $\tau_n(G)=\tau_n(H)$.
\end{theorem}
Subgroups of (two-sided) configuration equivalent groups have not been studied yet. However, By the use of  Tarski number, some pieces of information about non-Abelian free subgroups can be obtained. By a theorem of J{\'o}nsson and Dekker (see, for example, \cite[Theorem 5.8.38]{M.sapir}), $\tau(G) = 4$ if and only if $G$ contains a non-Abelian free subgroup. This and \cite[Theorem 1]{try}, implies that if groups $G$ and $H$ are configuration equivalent, then they both have non-Abelian free subgroups or not. In fact, we have more, and before stating the result, some reminders are needed; The power function for a polynomial type group $G$ is a function $\varsigma$ on $G$ with following properties:
\begin{enumerate}\label{1to3}
\item $\varsigma(G)$ is a finite set, and 
\item $\varsigma^{-1}(\varsigma(e_G))=\{e_G\}$.
\end{enumerate}
We associate a (finite) partition with a power function, called $\varsigma$--partition of $G$ consisting of disjoint subsets, $E(\varsigma(g)):=\varsigma^{-1}(\varsigma(g))$, $g\in G$. 
\begin{enumerate}
\item[(3)] There is a generating set $\g$ of $G$ which, for each $g\in G$, we can find a representative pair $(J_g,\sigma_g)$ in $\g$ such that if $\con gE=\con hF$, then 
$$W(J_g,\sigma_g;\h) F(\varsigma(e_G))\subseteq F(\varsigma(W(J_g,\sigma_g;\g))$$
where $F(\varsigma(g))$ considered to be the same as $E(\varsigma(g))$ in color. 
\end{enumerate}
By use of the power function, we can speak about some subgroups of configuration equivalent groups; 
  \begin{theorem}
Let two groups $G$ and $H$ be configuration equivalent. Suppose that $G$ contains a non-Abelian free subgroup of rank $n$, $n\in\mathbb N$. Then $H$ contains a non-Abelian free subgroup of rank $n$.
\end{theorem}
\begin{proof}
Fix a generating set of $G$, say $\g$. Let $\g_0$ be a generating set of a subgroup $G_0$ of $G$, which is free of rank $n$. Assume that $\varsigma$ is the power function, and $\E_0$ is a $\varsigma$-partition of $G_0$. Assume $\A$ is the sigma algebra generated by set $E(\varsigma(g))$, $g\in G_0$. Now, if 
$\Con(\g_0\oplus\g,\A)=\Con(\h_0\oplus\h,\B)$, and if we denote the set $F$ in $\F$, where $F\circledS E(\varsigma(g))$, by $F(\varsigma(g))$, then one can easily check that, for an arbitrary representative pair in $\g_0$, say $(J,\rho)$, we have 
$$W(J,\rho;\h_0)F(\varsigma(e_G))\subseteq F(\varsigma(W(J,\rho;\g_0))$$
Hence, $W(J,\rho;\h_0)=e_H$ if and only if $W(J,\rho;\g_0)=e_G$, and this means $H_0:=\langle\h_0\rangle$, is a subgroup of $H$, which is free of rank $n$.
\end{proof}
In the following conjugacy classes are involved; For an element $g$ in a group $G$, we denote the conjugacy class of $g$ by $\Cl g$. It is obvious that a normal subset (see \cite[Definition 2.1.]{rs1}) of $G$ is nothing but a disjoint union of conjugacy classes of $G$. We will show in the following theorem that, if two groups $G$ and $H$ are two-sided configuration equivalent, then their class numbers will be the same:
 \begin{theorem}
 \label{class number}
 Let finitely generated groups $G$ and $H$ be two-sided configuration equivalent. Then their class numbers are the same. 
 \end{theorem}
  \begin{proof}
Let $\g=(g_1,\dotsc,g_n)$ be a generating set of $G$. Suppose the class number of $G$ is at least $N$, $N\in\mathbb N$. So, there are elements $x_i$, $i=1,\dotsc, N$ in $G$, such that $\Cl{x_i}$, $i=1,\dotsc, N$, are pairwise disjoint. Consider a finite sigma algebra $\A$ of $G$, containing the following sets, 
$$\Cl{x_i}, \Cl{x_i}g_j^{-1}, \quad i=1,\dotsc,N,\, j=1,\dotsc,n$$
Assume $\B$ is a sigma algebra of $G$ along with a generating set $\h$, such that $\tcon gA=\tcon hB$. If elements $B_i$ in $\B$, $i=1,\dotsc,N$ are such that $B_i\circledS\Cl{x_i}$, then we get by Lemma \ref{important} that $B_i$, $i=1,\dotsc,N$, are normal sets, so the class number of $H$ is at least $N$. This completes the proof.
\end{proof}
It is obvious that $\Cl x=\{x\}$ if and only if $x$ is a central element. Suppose that the class number of $G$ is finite, this implies that $Z(G)$ is a finite subgroup of $G$, where $Z(G)$ stands for the center of $G$. In the below lemma, by using of the finiteness of the class number, we present a certain type of configuration pair, which make it possible to study subgroups  more efficiently:
\begin{lemma}
\label{C.C.P}
Let $G$ be a finitely generated group with finite class number. Fix a generating set $\g$ of $G$. Then there is a partition $\E$ of $G$, containing $\{e_G\}$, such that if $\tcon gE=\tcon hF$ for a configuration pair $(\h,\F)$ of a group $H$, and if $\{e_G\}\circledS F\in\F$, then $F$ will be singleton, and without loss of generality, we can assume that $F=\{e_H\}$. 
\end{lemma}
\begin{proof}
Suppose that $\g=(g_1,\dotsc,g_n)$, and $Z(G)=\{z_0=e_G,z_1,\dotsc,z_r\}$, Let elements $x_1,\dotsc,x_s\in G$ be such that 
$$\Cl{z_i},\Cl{x_j},\quad i=0,1,\dotsc,r,\, j=1,\dotsc,s$$
are all of the conjugacy classes. Now, consider the sigma algebra $\A$ generated by the following subsets of $G$, 
$$
\{z_i\}, \{z_ig_k^{-1}\}, \{x_j\},\Cl{x_j}\setminus\{x_j\}, \Cl{x_j}g_k^{-1}
$$
for $i=0,1,\dotsc,r$, $j=1,\dotsc,s$, and $k=1,\dotsc,n$. If $\tcon gA=\tcon hB$ for a sigma algebra $\B$ and a generating set $\h$ of a group $H$, and if $K_i$ and $L_j$ in $\B$ are such that 
$$K_i\circledS\{z_i\}, L_j\circledS\Cl{x_j}\quad i=0,1,\dotsc,r,\, j=1,\dotsc,s$$ 
 These sets are all normal (see Lemma \ref{important}), so by Theorem \ref{class number}, there are elements $w_i$, and $y_j$, in $H$ such that 
$$K_i=\Cl{w_i}, L_j=\Cl{y_j}, \quad i=0,1,\dotsc,r,\, j=1,\dotsc,s$$
Since sets $\Cl{y_j}$ can be written as a union of at least two atoms, so $w_i$ are all central, and without loss of generality we can assume $w_0=e_H$. Therefore $\F=\atom B$ works well. 
\end{proof}
By the power function of polynomial type groups, we can say that in two-sided configuration equivalent groups with finite class number, the polynomial type subgroups are isomorphic:
\begin{theorem}
\label{polysubgroups}
Let two groups $G$ and $H$ be two-sided configuration equivalent. Assume that these groups have finite class number. Then for each polynomial type subgroup, $G_0$ of $G$, $H$ contains a subgroup $H_0$, which is isomorphic to $G_0$. Furthermore $G$ and $H$ have isomorphic centers.
\end{theorem}
\begin{proof}
Fix a generating set $\g$ for $G$. Let $\varsigma$, $\E_0$ and $\g_0$ be such that the properties (1) to (3) on page in page \pageref{1to3}, are established for $G_0$. Applying Lemma \ref{C.C.P}, we get a partition $\E$ of $G$. Now. if 
$$\Con_{\text t}(\g_0\oplus\g,\sigma(\E_1\cup\E))=\Con_{\text t}(\h_0\oplus\h,\B)$$
 by Lemma \ref{C.C.P}, property (3), one can easily see that $G_0\cong H_0$, where $H_0:=\langle \h_0\rangle$.\par 
Since $Z(G)$ becomes a finite subgroup, for $G$ has a finite class number, by adding it to $\g$, repeating the proof of \ref{C.C.P}, one can easily obtain that $Z(G)\cong Z(H)$.
\end{proof}
It is of interest to study finitely generated groups with finite class numbers. For each large prime number $p$, there exists a 2-generated infinite group of exponent $p$ which has exactly $p$ conjugacy classes (that's \cite[Theorem 41.2]{Olshansky}). Osin \cite{osin} has recently constructed a finitely generated example (that was a major breakthrough since the problem was open for more than 60 years). All of these groups are infinitely presented and it is still an open problem, i.e., the question whether there is an infinite finitely presented group with a finite class number. If such a group exists, the following theorem will make sense:
\begin{theorem}
\label{f.p with finite class number} 
Let finitely presented group $G$ has finite class number. Suppose that $G\approx H$, which $H$ is a finitely generated group. Then $H$ is a quotient of $G$.  
\end{theorem}
\begin{proof}
Fix a generating set of $G$, say $\g$. Apply Lemma \ref{C.C.P} to get a partition $\E$ of $G$. Consider a set of defining relators 
$$\{W(J_i,\rho_i;\g):\,i=1,\dotsc,m\}$$
Let $\A$ be a sigma algebra of $G$, generated by sets in $\E$ and 
$$\{W(J,\rho;\g)\}\,\quad \n J\leq\max\{\n{J_i}:i=1,\dotsc,m\}$$
But $G\approx H$, so there is a sigma algebra $\B$ in $H$, such that $\tcon gA=\tcon hB$. Lemma \ref{partitions in sigma algebras}, Lemma \ref{C.C.P}, imply that we can assume $\{e_H\}\in\B$, $\{e_G\}\circledS\{e_H\}$, and 
{\small $$\{W(J,\rho;\g)\}\circledS\{W(J,\rho;\h)\}, \text{ for $(J,\rho)$ with}\, \n J\leq\max\{\n{J_i}:i=1,\dotsc,m\}$$}
so, in particular, $W(J_i,\rho_i;\h)=e_H$, $i=1,\dotsc,m$, and hence the map
$$G\rightarrow H,\quad W(I,\delta;\g)\mapsto W(I,\delta;\h)$$
where $(I,\delta)$ ranges over arbitrary representative pairs, introduced an epimorphism. 
\end{proof}

\section{Two-sided Configuration and Quotients}
In \cite{rs1}, some results about quotients of two-sided configuration equivalent groups were obtained. In this section, we study the number of quotients. \par Let $N$ be a normal subgroup of $G$; We denote the quotient map, $G\rightarrow G/N$ by $q_N$, and in the cases where there is no ambiguity, we may drop $N$. Recall that if $\E$ is a partition of $G/N$, then 
\[q^{-1}(\E):=\{q^{-1}(E):E\in\E\}\]
becomes a partition of $G$. We refer to such a partition, when we say an \textit{$N$--partition} of $G$. One can easily check that the intersection of two $N$--partitions is itself an $N$--partition. An $N$--refinement of an $N$--partition $\hat\E$, is a refinement, $\hat{\E'}$, which is itself an $N$--partition. Let $\g=(g_1,\dotsc,g_k)$ be an ordered subset of $G$, such that $q(\g):=( q(g_1),\dotsc,q(g_k))$ is a generating set of $G/N$, then by \cite[Lemma 6.2.]{arw}, there is an ordered subset $\mathfrak{n}$ of $N$, such that $\hat{\g}=\g\oplus\mathfrak{n}$ becomes a generating set of $G$, called \textit{$N$-extension of} $\g$. By an \textit{$N$--configuration pair} we mean a configuration pair $\Ncp gE$ such that $\g$ is an $N$--extension generating set, and $\hat\E$ is an $N$--partition of $G$. \\
The notions of  \enquote{preserving presentation} and \enquote{recognizability} of a configuration pair are defined in \cite{rs1}. Let’s first consider the concept of recognizable configuration pair:
\begin{definition}
\label{recognizable}
Let $G$ be a group with a normal subgroup $N$ and an $N$--configuration pair $\Ncp gE$. We may say that $\Ncp gE$ is \textit{recognizable} w.r.t. $N$, if whenever $\tcon {\hat g}E=\tcon {\hat h}{\hat F}$, for a configuration pair $\Ncp hF$ of a groups $H$, then for every $g\in G\setminus N$, there is a representative pair, $(J_g,\rho_g)$ such that 
$$g=W(J_g,\rho_g;\hat\g)\quad\text{and}\quad W(J_g,\rho_g;\hat\h)F\cap F=\nil$$
 where $F\in\F$ is in the same color as an element of $\E$ which contains $N$.
 \end{definition}
 As a consequence of Lemma \ref{partitions in sigma algebras}, one can easily show that (see \cite[Lemma 3.3.]{rs}):
\begin{lemma}
\label{refinement}
Let $G$ be a group with a normal subgroup $N$. Assume $\Ncp gE$ is a recognizable configuration pair w.r.t. $N$. Then for each $N$--refinement $\hat{\E'}$ of $\hat\E$, the configuration pair $\Ncp g{E'}$ is recognizable w.r.t. $N$. 
\end{lemma}
  What really makes working with recognizable configuration pair useful, is accompaniment of this notion to the following one:
 \begin{definition}
\label{pp}
Let $G$, $N$, $\g$, $\hat{\g}$ be regarded as above. We say that a configuration pair $(\hat{\g},\E)$ \textit{preserves presentation} w.r.t. $N$, if the following are held:
\\(I) $N\in\E$,
\\(II) if $\tcon {\hat g}{\hat E}=\tcon {\hat h}{\hat F}$, for a configuration pair $\Ncp hF$ of a groups $H$, and $N\circledS F\in\hat\F$, then $F$ is normal and the function bellow, defined on cosets of $N$, 
 $$W(J,\rho;\hat{\g})N\mapsto W(J,\rho;\hat\h)F$$
 is considered to be well-defined, where $(J,\rho)$ is an arbitrary representative pair.\\
\end{definition}
\begin{remark}
\label{refinement}
By \cite[Theorem 2.5]{rs1}, we know that if $G/N$ is finitely presented, then for each $N$--configuration pair $\Ncp gE$, there exists an $N$--refinement $\hat{\E'}$ of $\hat\E$, such that $\Ncp g{E'}$ preserves presentation w.r.t. $N$. 
\end{remark}
\par Let $\R$ and $\S$ be two partitions of a group $G$. Define the partition $\R\cap\S$ as follows:
 $$\atom{\sigma(\R\cup\S)}$$
 regarding this notation and Lemma \ref{partitions in sigma algebras} as well, we can see that it does not matter, if we work with one or more than one partition of $G$, more precisely:
 \begin{lemma}
\label{finite partitions}
Let $G$ be a group with a generating set $\g$. Assume that $\{\R_1,\dotsc,\R_r\}$, be a collection of partitions of $G$ and $G\approx_{\text t}H$ for a group $H$. Then, there are, a generating set $\h$, and a collection $\{\S_1,\dotsc,\S_r\}$ of partitions of $H$, such that 
$$\Con_{\text t}(\g,\R_i)=\Con_{\text t}(\h,\S_i)\quad i=1,\dotsc,r$$
\end{lemma}
\begin{proof}
Set $\R=\bigcap_{i=1}^r\R_i$. Let $(\h,\S)$ be a configuration pair of $H$ such that $\tcon gR=\tcon hS$. If $\S_i$ is a partition of $H$, such that $(\R,\R_i)\sim(\S,\S_i)$, $i=1,\dotsc,r$. Then, Lemma \ref{partitions in sigma algebras}, implies that $\Con_{\text t}(\g,\R_i)=\Con_{\text t}(\h,\S_i)$, $i=1,\dotsc,r$.
\end{proof}
With the concept of two-sided configuration, we can study the number of finite index subgroups; To this end, we provide the following definition:
\begin{definition}\label{r.w.r.t.C}
Let $G$ be a group, and $\mathcal C$ be a finite collection of normal subgroups of $G$. We may say $G$ is recognizable w.r.t $\mathcal{C}$, if the following are held:
\begin{enumerate}
\item $\mathcal{C}$ is closed under intersection,
\item $G/N$ is finitely presented, for all $N\in\mathcal C$, and
\item There exists a generating set $\g$ of $G$, and a collection $\{\E(N):\, N\in\mathcal{C}\}$ of partition of $G$ such that $(\g,\E(N))$ is a recognizable configuration pair w.r.t. $N$.
\end{enumerate}
\end{definition}
In the case which there is a collection $\mathcal{C}$, with properties in the above definition, the next theorem is worthy of attention:
\begin{theorem}\label{norml correspondence}
Suppose a group $G$ is recognizable w.r.t. a collection $\mathcal C$ of its normal subgroups. Let a group $H$ be two-sided configuration equivalent with $G$. Then there is a collection $\{\mathfrak{N}_N:\,N\in\mathcal C\}$ of normal subgroups of $H$ such that 
\begin{enumerate}
\item $\frac{G}{N}\cong\frac{H}{\mathfrak N_N}$,
\item $\mathfrak N_{N\cap M}=\mathfrak N_N\cap\mathfrak N_M$, for $N$ and $M$ in $\mathcal C$.
\end{enumerate}
\end{theorem}
\begin{proof}
Let $\g$ and $\{\E(N):\,N\in\mathcal C\}$ be regarded as in Definition \ref{r.w.r.t.C}. By Lemma \ref{finite partitions}, there are a generating set $\h$ and partitions $\{\F(N):\,N\in\mathcal C\}$, such that 
$$ \Con_\text t(\g,\E(N))=\Con_\text t(\h,\F(N))\quad(N\in\mathcal C)$$
Hence, by \cite[Lemma 3.4.]{rs1}, we obtain normal subgroups $\{\mathfrak{N}_N:\,N\in\mathcal C\}$, along with following isomorphisms
\[ \frac{G}{N}\rightarrow\frac{H}{\mathfrak N_N},\quad \word J\rho\g N\mapsto\word J\rho\h\mathfrak{N}_N\]
With the above isomorphisms, the equation in (2) are easily proved.
\end{proof}
We know by \cite[Theorem 21.4]{intro} that, the number of subgroups of finite index $n$ in a finitely
generated group $G$ is finite. Since the intersection of finite index subgroups is again a finite index subgroup, by the above theorem, we obtain:
\begin{corollary}
Let $G$ and $H$ be finitely generated groups with the same two-sided configuration sets, and let $n\in\mathbb N$. Then $G$ and $H$ contain exactly the same number of normal subgroups of index $n$. Moreover, we have 
$$\prod\{G/N: |G:N|<\infty\}\cong\prod
\{H/\mathfrak N: |H:\mathfrak N|<\infty\}$$
\end{corollary}
\begin{proof}
Suppose $\{N_1,\dotsc,N_r\}$ is the collection of all normal subgroups of index $n$. Let $\mathcal C$ be a collection of normal subgroups of $G$ obtained from intersection of $N_k$'s. For $N\in\mathcal C$, assume that the power function of $G/N$ yields a configuration pair $(q_N(\g_N),q_N(\E(N)))$, for  an ordered set $\g_N$ and partition $\E(N)$ of $G$. Let $\g=\widehat{\oplus_{k=1}^r\g_{N_k}}$ be a $\bigcap_{k=1}^rN_k$--extension generating set of $G$. Hence, one can see, by $\g$ and $\mathcal C$ as above, the conditions of Definition \ref{r.w.r.t.C} are satisfied, and therefore, Theorem \ref{norml correspondence}, shows the existence of at least $r$ normal subgroups of index $n$ in $H$. The symmetry in the concept of configuration completes the proof.
\end{proof}
We denote by $\textbf{P}(G)$, the collection of normal subgroups of $G$, which their quotient in $G$ are polycyclic. Since every polycyclic group has a recognizable configuration pair, and since $\textbf{P}(G)$ is closed under intersection, an argument like the one in proof of the previous corollary leads to:
\begin{corollary}
Let $G$ and $H$ be finitely generated groups with $G\approx_\text tH$. Then $\textbf{P}(G)$ and $\textbf{P}(H)$ have the same cardinality. Furthermore if $\textbf P(G)$ is finite, then there is a bijection $\Psi:\textbf P(G)\rightarrow\textbf P(H)$ such that 
\begin{enumerate}
\item For every $N$ and $M$ in $\textbf P(G)$, $\Psi(N\cap M)=\Psi(N)\cap\Psi(M)$,
\item For each $N\in\textbf P(G)$, we have $G/N\cong H/\Psi(N)$.
\end{enumerate}
\end{corollary}

\section{The Concept of Configuration Equivalence is not Equivalent to Isomorphism}
The question whether the concepts of the configuration equivalence and isomorphism of groups are the same has been open since the beginning of configuration theory. Indeed, the answer is \enquote{No}, there are non-isomorphic groups with the same two-sided configuration sets. It is worth noting that, the groups in our example are solvable, thus this natural conjecture that these two concepts may be equivalent for amenable groups will be rejected. To provide our example, we may need to provide the following two technical lemmas:
\begin{lemma}
\label{epimorphism and configuration}
Let $\phi:G\rightarrow H$ be an epimorphism of groups. Suppose that $\g$ is a generating set of $G$, and $\F$ is a partition of $H$, then 
$$\Con_\text t(\g,\phi^{-1}(\F))=\Con_\text t(\phi(\g),\F)$$
\end{lemma}
\begin{proof}
Assume that $C\in\Con_\text t(\g,\phi^{-1}(\F))$. So, there exists {\small $(x_0,x_1,\dotsc,x_{2n})$} having configuration $C$, that means, $x_k\in \phi^{-1}(F_{c_k})$, $k=0,1,\dotsc,2n$ and, 
$$x_k=g_kx_0,\quad x_{k+n}=x_0g_k\quad(k=1,\dotsc,n)$$
Therefore $\phi(x_k)\in F_{c_k}$, and 
$$\phi(x_k)=\phi(g_k)\phi(x_0),\quad \phi(x_{k+n})=\phi(x_0)\phi(g_k)$$
 Hence, $(\phi(x_0),\phi(x_1),\dotsc,\phi(x_{2n}))$ has configuration  $C$ in $\Con_\text t(\phi(\g),\F)$. 
\par Conversely, let a tuple $(y_0,y_1,\dotsc,y_{2n})$ have configuration {\small $C\in\Con_\text t(\phi(\g),\F)$}. That means, $y_k\in F_{c_k}$, $k=0,1,\dotsc,2n$ and 
$$y_k=\phi(g_k)y_0,\quad y_{k+n}=y_0\phi(g_k)
\quad (k=1,\dotsc,n)$$
Choose $x_0\in\phi^{-1}(y_0)$. Then 
$$(x_0,g_1x_0,\dotsc,g_nx_0,x_0g_1,\dotsc,x_0g_n)$$
in $G$ has configuration $C$, belonging to $\Con_\text t(\g,\phi^{-1}(\F))$. 
\end{proof}
We may call an homomorphism $\phi:G\rightarrow H$, \textit{generating-surjection}, if every ordered generating set of $H$ is an image of an ordered generating set of $G$. It is clear that each generating-surjection, becomes an epimorphism. We call two groups $G$ and $H$ \textit{generating-bijective}, if there are generating-surjections $\phi:G\rightarrow H$, and $\psi:H\rightarrow G$. From the following lemma, it will be understood that generating-bijective groups have same two-sided configuration sets:
\begin{lemma}
\label{technical}
Let $G$ and $H$ be two finitely generated groups. Suppose there is a generating-surjection $\phi:G\rightarrow H$. Then $\Con_\text t(H)\subseteq\Con_\text t(G)$.
\end{lemma}
\begin{proof}
Consider a configuration pair $(\h,\F)$ of $H$. There exists a generating set $\g$ of $G$ such that $\h=\phi(\g)$. Thus, by previous lemma,
$$\tcon hF=\Con_\text t(\phi(\g),\F)=\Con_\text t(\g,\phi^{-1}(\F))$$
This completes the proof.
\end{proof}
It is obvious that isomorphic groups are generating-bijective, but the converse is not always true, as the following theorem will show this:
\begin{theorem}\label{non-isomorphic generatin-bijective groups}
There exist non-isomorphic finitely generating groups with the same two-sided configuration sets.
\end{theorem}
\begin{proof}
Put $\mathbf R:=\mathbb Z[t,t^{-1}]$, the ring of Laurent polynomials. Let $K$ be a group of matrices
\[\begin{pmatrix}
1&B&D\\
0&A&C\\
0&0&1
\end{pmatrix}
\]
where $B$, $C$ and $D$ belongs to $\mathbf{R}$, and $A\in\langle t\rangle$. The group $B$ is easily checked to be finitely generated; Indeed, if we denote the above matrix by $(A,B,C,D)$, then $(\mathbf k_1,\mathbf k_2,\mathbf k_3)$, in which 
 \[ \mathbf k_1=(t,0,0,0)\quad \mathbf k_2=(1,1,0,0)\quad \mathbf k_3=(1,0,1,0)\]
  becomes a generating set of $K$, because of the following equations:
  \begin{align*}
 \mathbf k_1^{m}&=(t^k,0,0,0), \quad \mathbf k_1^{-m}\mathbf k_2\mathbf k_1^k=(1,t^m,0,0)\\
 \mathbf k_1^m\mathbf k_3\mathbf k_1^{-m}&=(1,0,t^m,0),\quad [\mathbf k_3,\mathbf k_1^{-m}\mathbf k_2\mathbf k_1^m]=(1,0,0,t^m)
 \end{align*}
 \noindent where $m$ is an integer. The center of this group, $Z(K)$, consists of unipotent matrices with a single possibly non-trivial element in the upper right corner
\[ Z(K)=\{(1,0,0,D):\, D\in\mathbf R\}\] 
It is clearly isomorphic to $\mathbf{R}$. We can rewrite the product of the group as bellow:
 \[ (A,B,C,D)(X,Y,Z,W)=(AX,BX+Y,C+AW,D+BW+Z)\]
 By the above equality, one can see that the map
 \begin{equation*}
\Phi: K\rightarrow K,\quad 
(A,B,C,D)\mapsto(A,B,tC,tD)
 \end{equation*} 
 introduced an automorphism of $K$. 
\par Set $N_m:=t^m\mathbb Z[t]$, $m\in\mathbb Z$. The automorphism $\Phi$, implies that $K/N_m\cong K/\Phi(N_m)=K/N_{m+1}$, $m\in\mathbb Z$. Now, let $G:=K/N_0$ and $H:=K/(2\mathbb Z\oplus N_1)$. These groups are not isomorphic, for $G$ is torsion free and $H$ is not.
\par Note that if $\mathfrak k$ is an ordered set of $K$, such that its image under the natural quotient map forms a generating set of $K/N_m$, then, also, the image of $\Phi(\mathfrak{k})$ will form a generating set of $K/N_m$. Thus, $G$ and $H$ are generating-bijective. This and Lemma \ref{technical}, complete the proof. 
\end{proof}
\begin{remark}
The groups $G$ and $H$ in the above theorem are solvable. this shows that the two concepts, configuration equivalence and isomorphism, are not equivalent for solvable and hence amenable groups. these groups are not finitely presented and they are not residually finite, for they are not Hopfian. 
\end{remark}

\end{document}